\documentclass[11pt]{article}
\usepackage{amsfonts,amsmath,amsthm}

\usepackage[usenames]{color}
 
\usepackage[colorlinks,%
linkcolor=BrickRed,%
filecolor =BrickRed,%
citecolor=RoyalPurple,%
]{hyperref}

\newcommand{\Xp}{\mbox{\boldmath $X$}}

\newcommand{\IND}{{\mathbb I}}

\newcommand{\la}{\lambda}

\def\be{\begin{eqnarray}}
\def\ee{\end{eqnarray}}
\def\ben{\begin{eqnarray*}}
\def\een{\end{eqnarray*}}

\def\elabel#1{\label{e:#1}}

\def\sq{$\Box$}

\def\qed{\ifmmode\sq\else{\unskip\nobreak\hfil
\penalty50\hskip1em\null\nobreak\hfil\sq
\parfillskip=0pt\finalhyphendemerits=0\endgraf}\fi\par\medbreak}

\newsavebox{\junk}
\savebox{\junk}[1.6mm]{\hbox{$|\!|\!|$}}
\def\lll{{\usebox{\junk}}}

\def\limsup{\mathop{\rm lim\ sup}}

\def\state{{\sf X}}

\newcommand{\field}[1]{\mathbb{#1}}

\def\nat{\field{Z}_+}

\def\Co{\field{C}}

\def\ind{\field{I}}

\newcommand{\one}{\hbox{\rm\large\textbf{1}}}

\def\bfPhi{\mbox{\protect\boldmath$\Phi$}}

\def\haP{{\widehat P}}

\def\til={{\widetilde =}}

\def\tilP{{\widetilde P}}

\def\clB{{\cal B}}

\def\clS{{\cal S}}

\def\half{{\mathchoice{\textstyle \frac{1}{2}}%
{\frac{1}{2}}%
{\hbox{\tiny $\frac{1}{2}$}}%
{\hbox{\tiny $\frac{1}{2}$}} }}

\def\eqdef{\mathbin{:=}}

\def\Expect{{\sf E}}

 \def\eq#1/{(\ref{#1})}

\def\epsy{\varepsilon}

\def\varble{\,\cdot\,}

\newtheorem{theorem}{Theorem}[section]

\newtheorem{proposition}[theorem]{Proposition}
\newtheorem{lemma}[theorem]{Lemma}

\def\Lemma#1{Lemma~\ref{t:#1}}
\def\Proposition#1{Proposition~\ref{t:#1}}
\def\Theorem#1{Theorem~\ref{t:#1}}

\def\eq#1/{(\ref{e:#1})}

\newcommand{\beqn}[1]{\notes{#1}%
\begin{eqnarray} \elabel{#1}}

\newcommand{\eeqn}{\end{eqnarray} }

\newcommand{\beq}[1]{\notes{#1}%
\begin{equation}\elabel{#1}}

\newcommand{\eeq}{\end{equation}} 

\def\bdes{\begin{description}}
\def\edes{\end{description}}

\def\notes#1{}

\def\LV{L_\infty^V}

\def\FRAC#1#2#3{\genfrac{}{}{}{#1}{#2}{#3}}

\def\half{{\mathchoice{\FRAC{1}{1}{2}}%
{\FRAC{1}{1}{2}}%
{\FRAC{3}{1}{2}}%
{\FRAC{3}{1}{2}}}}

\newcounter{rmnum}
\newenvironment{romannum}{\begin{list}{{\upshape (\roman{rmnum})}}{\usecounter{rmnum}
\setlength{\leftmargin}{24pt}
\setlength{\rightmargin}{16pt}
\setlength{\itemindent}{-1pt}
}}{\end{list}}

\newcounter{anum}
   
\newenvironment{alphanum}{\begin{list}{{\upshape (\alph{anum})}}{\usecounter{anum}
\setlength{\leftmargin}{28pt}
\setlength{\rightmargin}{16pt}
\setlength{\itemindent}{0pt}
}}{\end{list}}

\newlength{\noteWidth}
\long\def\notes#1{\ifinner
             {\tiny #1}
             \else
             \marginpar{\parbox[t]{\noteWidth}{\raggedright\tiny #1}}
             \fi}
 
\newcounter{tasks}

{\begin{quote}%
\begin{list}{\hspace{-.75cm}\hbox{\textrm{\rm (\roman{tasks})}\ }}{\usecounter{tasks}%
        \setlength{\labelsep}{0pt}
        \setlength{\leftmargin}{0pt}
        \setlength{\rightmargin}{0pt}
        \setlength{\labelwidth}{0pt}
        \setlength{\listparindent}{0pt}}}%
{\end{list}\end{quote}}

\newcounter{tasksA}

{\begin{quote}%
\begin{list}{\hspace{-.75cm}\hbox{\textrm{\rm (\alph{tasksA})}\ }}{\usecounter{tasksA}%
        \setlength{\labelsep}{0pt}
        \setlength{\leftmargin}{0pt}
        \setlength{\rightmargin}{0pt}
        \setlength{\labelwidth}{0pt}
        \setlength{\listparindent}{0pt}}}%
{\end{list}\end{quote}}
 
\begin{document}

\title{\vspace{-2cm}%
Geometric Ergodicity and the Spectral Gap\\
of Non-Reversible Markov Chains
}
\author{
      	I. Kontoyiannis\thanks{Department of Informatics,
		Athens University of Economics and Business,
		Patission 76, Athens 10434, Greece.
                Email: {\tt yiannis@aueb.gr}.
                 \newline          I.K. was supported in part
                by a Sloan Foundation Research Fellowship.}
\and
        S.P. Meyn\thanks{Department of Electrical and Computer 
                Engineering and the Coordinated Sciences Laboratory, 
                University of Illinois at Urbana-Champaign,  
                Urbana, IL 61801, USA. Email: {\tt meyn@uiuc.edu}.
                \newline
                S.P.M. was supported in part by the National Science Foundation
  ECS-0523620. Any opinions, findings,
  and conclusions or recommendations expressed in this material are
  those of the authors and do not necessarily reflect the views of the
  National Science Foundation.}
}

\maketitle
 
\begin{abstract} 
We argue that the spectral theory 
of non-reversible Markov chains 
may often be more effectively cast within
the framework of the naturally associated
weighted-$L_\infty$ space $L_\infty^V$, 
instead of the usual Hilbert 
space $L_2=L_2(\pi)$, where $\pi$ is
the invariant measure of the chain.
This observation is, in part,
based on the following results.
A discrete-time Markov chain with values
in a general state space is geometrically 
ergodic if and only if its transition kernel 
admits a spectral gap in $L_\infty^V$. 
If the chain is reversible, 
the same equivalence holds with $L_2$ in place 
of $L_\infty^V$, but in the absence of 
reversibility it fails: There are
(necessarily non-reversible, geometrically
ergodic) chains that admit a spectral gap 
in $L_\infty^V$ but not in $L_2$. Moreover,
if a chain admits a spectral gap in $L_2$,
then for any $h\in L_2$ there exists a Lyapunov
function $V_h\in L_1$ such that $V_h$ dominates $h$ 
and the chain admits a spectral gap 
in $L_\infty^{V_h}$. The relationship
between the size of the spectral gap
in $L_\infty^V$ or $L_2$, and the rate 
at which the chain converges to equilibrium 
is also briefly discussed.

%
%
 
\bigskip

\bigskip


\bigskip







\end{abstract}


\thispagestyle{empty}

\newpage
 
\setcounter{page}{1}
\section{Introduction and Main Results}

There is increasing interest in spectral theory and rates of convergence 
for Markov chains.  Research is motivated by elegant mathematics as well 
as a range of applications.
In particular, one of the most effective general
methodologies used to establish bounds on the convergence
rate of a geometrically ergodic chain is via an analysis
of the spectrum of the chain's transition kernel.
See, e.g.,
\cite{meyn-tweedie:94b,ros95a,ros95b,diaconis-SC:98,%
deljun99,kontoyiannis-meyn:I,huimeysch04a,kontoyiannis-meyn:II,%
winkler:book,roberts-rosenthal:97,brooks-roberts:98,%
bremaud:book,tetali:06,diaconis:07,%
diaconis:08,diaconis:09},
and the relevant references therein.

The word {\em spectrum} naturally invites techniques grounded in a 
Hilbert space framework. The majority of quantitative results 
on rates of convergence are obtained using such methods, 
within the Hilbert space $L_2=L_2(\pi)$, where $\pi$ denotes 
the stationary distribution of the Markov chain in question.
Indeed, most successful studies have been carried out 
for Markov chains that are reversible, in which case
a key to analysis is the fact that
the transition kernel, viewed as a linear operator on $L_2$, 
is self-adjoint.  In this paper 
we argue that, in the {\em absence} of reversibility, the Hilbert space 
framework may not be the appropriate setting for
spectral analysis.

To be specific, let $\Xp=\{X(n)\;:\;n\geq 0\}$ denote a 
discrete-time Markov chain with values on a general
state space $\state$.  We assume that $\state$ 
is equipped with a countably generated sigma-algebra
$\clB$. The distribution of $\Xp$ 
is described by its initial
state $X(0)=x_0\in\state$ and
the transition semigroup
$\{P^n\,:\, n\geq 0\}$, where, for each $n$,  
\[
P^n(x,A)\eqdef\Pr\{X(n)\in A\,|\,X(0)=x\},
\quad x\in\state,\,A\in\clB.
\]
For simplicity we write $P$ 
for the one-step kernel $P^1$.
Recall that each $P^n$, like any
(not necessarily probabilistic)
kernel $Q(x,dy)$ acts on functions
$F:\state\to\Co$ and signed measures $\nu$
on $(\state,\clB)$, via,
\ben
QF(\cdot)=\int_{\state}Q(\cdot,dy)F(y)
\;\;\;\mbox{and}\;\;\;
\nu Q(\cdot)=\int_{\state}\nu(dx)Q(x,\cdot),
\een
whenever the integrals exist.
Throughout the paper, we assume that the chain $\Xp=\{X(n)\}$ 
is {\em $\psi$-irreducible and aperiodic}; cf.\
\cite{meyn-tweedie:book2,nummelin:book}. This means
that there is a $\sigma$-finite measure
$\psi$ on $(\state,\clB)$ such that, for any
$A\in\clB$ with $\psi(A)>0$, and any $x\in\state$,
$$P^n(x,A)>0,
\;\;\;\;\mbox{for all $n$ sufficiently large}.$$
Moreover, we assume that $\psi$ is {\em maximal}
in the sense that any other such $\psi'$ is absolutely
continuous with respect to $\psi$.

\subsection{Geometric ergodicity}
The natural class of chains to consider in the present
context is that of {\em geometrically ergodic} chains,
namely, chains with the property that there exists
an invariant measure $\pi$ on $(\state,\clB)$ and
functions $\rho:\state\to(0,1)$ and $C:\state\to[1,\infty)$,
such that,
$$\|P^n(x,\cdot)-\pi\|_{\rm TV}\leq C(x)\rho(x)^n,
\;\;\;\;\mbox{for all}\;n\geq 0,\;\pi\mbox{-a.e.}\;x\in\state,$$
where $\|\mu\|_{\rm TV}:=\sup_{A\in\clB}|\mu(A)|$
denotes the total variation norm on signed measures.
Under $\psi$-irreducibility and aperiodicity
this is equivalent \cite{meyn-tweedie:book2,roberts-rosenthal:97}
to the seemingly stronger
requirement that there is a single constant
$\rho\in(0,1)$, a constant $B<\infty$ and
a $\pi$-a.e.\ finite function $V:\state\to[1,\infty]$, 
such that,
\be
\|P^n(x,\cdot)-\pi\|_V\leq B\,V(x)\rho^n,
\;\;\;\;\mbox{for all}\;n\geq 0,\;\pi\mbox{-a.e.}\;x\in\state,
\label{eq:conv}
\ee
where $\|\mu\|_V:=\sup\{|\int F\,d\mu|\;:\;F\in L_\infty^V\}$
denotes the $V$-norm on signed measures,
and where $L_\infty^V$ denotes the weighted-$L_\infty$ 
space consisting of all measurable functions
$F\colon\state\to \Co$ with,
\begin{equation} 
\|F\|_V \eqdef \sup_{x\in\state}\frac{|F(x)|}{V(x)} < \infty\, .
\label{e:gV}
\end{equation}

Another equivalent and operationally simpler definition
of geometric ergodicity for a $\psi$-irreducible, aperiodic
chain $\Xp=\{X(n)\}$, is that it satisfies the following drift
criterion \cite{meyn-tweedie:book2}:
$$
\left. \begin{array}{ll}
   & \parbox{.6\hsize}{There is 
a function $V\colon\state\to[1,\infty]$,     
 a small set $C\subset\state$,
 \\
  and constants
                $\delta>0,\,$ $b<\infty$, such that:}\\
   & \\
   & \hspace{0.8in} PV\leq (1-\delta)V+b\IND_C\, .
 \end{array} \right\}
\hspace{0.7in}
\mbox{\bf (V4)}
$$ 
We then say that the chain is {\em geometrically ergodic
with Lyapunov function $V$}. In (V4) it is always 
assumed that the Lyapunov function $V$ is finite 
for at least one $x$ (and then it is necessarily 
finite $\psi$-a.e.).
Also, recall that a set
$C\in\clB$ is {\em small} if there exist $n\geq 1$, 
$\epsilon>0$ and a probability measure $\nu$ on 
$(\state,\clB)$ such that,
$P^n(x,A)\geq\epsilon\IND_C(x)\nu(A)$,
for all $x\in\state,$ $A\in\clB$.

Our first result relates geometric ergodicity
to the spectral properties of the kernel $P$.
Its proof, given at the end of Section~\ref{s:proofs},
is based on ideas from \cite{kontoyiannis-meyn:I}. 
See Section~\ref{s:spectrum} 
for more precise definitions.

\begin{proposition}
\label{t:equiv:V}
A $\psi$-irreducible and aperiodic Markov
chain $\Xp=\{X(n)\}$ is geometrically ergodic
with Lyapunov function $V$ if and only if 
$P$ admits a spectral gap in $L_\infty^V$.
\end{proposition}

\subsection{Reversibility}
Recall that the chain $\Xp=\{X(n)\}$ is called
{\em reversible} if there is a probability measure
$\pi$ on $(\state,\clB)$ satisfying the detailed
balance conditions,
$$\pi(dx)P(x,dy)=\pi(dy)P(y,dx).$$
This is equivalent to saying that the 
linear operator $P$ is self-adjoint on
the space $L_2=L_2(\pi)$ of 
(measurable) functions $F:\state\to\Co$
that are square-integrable under $\pi$,
endowed with the inner product
$(F,G)=\int FG^*\,d\pi$,
where `$*$' denotes the complex
conjugate operation.

The following result is the natural
analog of Proposition~\ref{t:equiv:V}
for reversible chains. Its proof,
given in Section~\ref{s:proofs},
is partly based on results in
\cite{roberts-rosenthal:97}.

\begin{proposition}
\label{t:equiv:2}
A reversible, $\psi$-irreducible and aperiodic Markov
chain $\Xp=\{X(n)\}$ is geometrically ergodic
if and only if $P$ admits a spectral gap in $L_2$.
\end{proposition}

\subsection{Spectral theory} 
The main question addressed in this paper is whether
the reversibility assumption 
of Proposition~\ref{t:equiv:2}
can be relaxed. In other words, 
whether the space $L_2$ can be used
to characterize geometric ergodicity
like $L_\infty^V$ was in Proposition~\ref{t:equiv:V}.
One direction is true without reversibility:
A spectral gap in $L_2$ implies that
the chain is ``geometrically ergodic in $L_2$''
\cite{roberts-rosenthal:97}\cite{roberts-tweedie:01},
and this implies the existence of a Lyapunov
function $V$ satisfying (V4) \cite{roberts-tweedie:01}.
Therefore, the chain is geometrically ergodic
in the sense of \cite{kontoyiannis-meyn:I},
where it is also shown that it must admit a central
gap in $L_\infty^V$. A direct, explicit construction
of a Lyapunov function $V_h$ is given in our first 
main result stated next, where quantitative
information about $V_h$ is also obtained.
It is proved in Section~\ref{s:proofs}.

\begin{theorem}
\label{t:L2givesV}
Suppose that a $\psi$-irreducible, aperiodic
chain $\Xp=\{X(n)\}$ admits a spectral gap in $L_2$.  
Then, for any $h\in L_2$, there is $\pi$-integrable
function $V_h$,
such that the chain is geometrically ergodic
with Lyapunov function $V_h$ and 
$h\in L_\infty^{V_h}$.
\end{theorem}

But the other direction may not hold in the
absence of reversibility. Based on earlier counterexamples
constructed by H\"aggstr{\"o}m \cite{hag05a,hag05b} 
and Bradley \cite{bra83}, in Section~\ref{s:proofs} 
we prove the following:

\begin{theorem}
\label{t:geoNotL2spectrum}
There exists a $\psi$-irreducible, aperiodic
Markov chain $\Xp=\{X(n)\}$ which is geometrically 
ergodic but does not admit a spectral gap in $L_2$.   
\end{theorem}

\subsection{Convergence rates}
The existence of a spectral gap is intimately connected
to the exponential convergence rate for a $\psi$-irreducible,
aperiodic Markov chain. For example, if the chain is reversible,
we have the following well-known, quantitative bound.
See Section~\ref{s:spectrum} for detailed definitions;
the result follows from the
results in \cite{roberts-rosenthal:97}, 
combined with Lemma~\ref{t:L2gap} given
in Section~\ref{s:spectrum}.

\begin{proposition}
\label{p:TV}
Suppose that a reversible chain $\Xp=\{X(n)\}$ is $\psi$-irreducible,
aperiodic, and has initial distribution $\mu$.
If the chain $\Xp$ admits a spectral gap $\delta_2>0$ in $L_2$,
then,
$$\|\mu P^n-\pi\|_{\rm TV}\leq
\frac{1}{2}\|\mu-\pi\|_2(1-\delta_2)^n,
\;\;\;\;n\geq 1,
$$
where the $L_2$-norm on signed measures $\nu$ is defined 
as the $L_2(\pi)$-norm of the density $d\nu/d\pi$ if it
exists, and is set equal to infinity otherwise.
\end{proposition}

In the absence of reversibility, 
the size of the spectral gap 
in $L_\infty^V$ precisely determines
the exponential convergence rate of
{\em any} geometrically ergodic
chain. The result of the following
proposition is stated in 
Lemma~\ref{t:LVgap} in Section~\ref{s:spectrum}.

\begin{proposition}
\label{t:uniform}
Suppose that the chain $\Xp=\{X(n)\}$ is $\psi$-irreducible
and aperiodic. If it admits a spectral gap $\delta_V>0$ in 
$L_\infty^V$,
then, for $\pi$-a.e. $x$,
$$\lim_{n\to\infty}\frac{1}{n}\log \|P^n(x,\cdot)-\pi\|_V 
= \log(1-\delta_V).$$
In fact, the convergence is uniform
in that:
\be \lim_{n\to\infty}\frac{1}{n}\log 
\left(\sup_{x\in\state,\;\|F\|_V=1}
\frac{|P^n F(x)-\int F\,d\pi|}{V(x)}
\right)
= \log(1-\delta_V).
\label{eq:uniform}
\ee
\end{proposition}

\medskip

Section~\ref{s:spectrum} contains precise
definitions regarding the spectrum and the 
spectral gap of the kernel $P$ acting either 
on $L_2$ or the weighted-$L_\infty$ space 
$L_\infty^V$. Simple properties of the 
spectrum are also stated and proved.
Section~\ref{s:proofs} contains the proofs 
of the first four results stated above.


\section{Spectra and Geometric Ergodicity}
\label{s:spectrum}

We begin by giving precise definitions
for the spectrum and spectral gap 
of the transition kernel $P$, viewed as
a linear operator. The spectrum depends 
on the domain of $P$, for which we consider 
two possibilities:
\begin{romannum}
\item   The Hilbert space $L_2=L_2(\pi)$, 
equipped with the norm $\|F\|_2=[\int |F|^2\,d\pi]^{1/2}$.
\item  The Banach space $\LV$, with norm $\|\varble\|_V $ 
defined in \eqref{e:gV}.
\end{romannum}
In either case, the spectrum is defined as the set of nonzero   
$\lambda\in\Co$  for which the inverse $(I \lambda -  P)^{-1}$  does 
not exist as a bounded linear operator on the domain of $P$.   
The transition kernel admits a \textit{spectral gap} if there 
exists $\epsy_0>0$ such that $\clS \cap \{ z : |z|\ge 1 - \epsy_0 \}$ 
is finite, and contains only poles of finite multiplicity;
see \cite[Section~4]{kontoyiannis-meyn:I} for more details.
The spectrum is denoted  $\clS_2 $  when $P$ is viewed as 
a linear operator on $L_2$,  and it is denoted  $\clS_V $  
when $P$ is viewed as a linear operator on $\LV$.   

The induced operator norm of a linear operator 
$\haP\colon\LV\to\LV$ is defined as usual via,
\ben
\lll \haP \lll_{V}
\eqdef  \sup \frac{\| \haP F\|_{V}}{\|F\|_{V}},
\een
where the supremum is over all $F\in L^V_\infty$
satisfying $\|F\|_{V}\neq 0$. An analogous definition
gives the induced operator norm $\lll \haP\lll_2$ of 
a linear operator
$\haP$ acting on $L_2$.

For a $\psi$-irreducible, aperiodic chain $\Xp=\{X(n)\}$,
geometric ergodicity expressed in the form (\ref{eq:conv})
implies that $P^n$ converges 
to a rank-one operator, at a geometric rate:
For some constants $B<\infty$, $\rho\in(0,1)$,
\begin{equation}
\lll P^n - \one\otimes \pi \lll_{V} \le  B\,\rho^n
        \, , \qquad n\geq 0,
\label{e:16.0.1}
\end{equation}
where the outer product $\one\otimes\pi$ denotes
the kernel $\one\otimes\pi(x,dy)=\pi(dy)$.
It follows that the inverse $[I\lambda - P +  \one\otimes \pi]^{-1}$ 
exists as a bounded linear operator on $\LV$,  
whenever $\lambda>\rho$.  This in turn implies that  
$P$ has a single isolated pole at    $\lambda= 1$ 
in the set $\{ \lambda\in\Co :  \lambda>\rho\}$, 
so that $P$ admits a spectral gap.  

In \Lemma{WhereIsS} we clarify the location 
of poles when the chain admits a spectral gap in $L_2$ or $\LV$.   

\begin{lemma}
\label{t:WhereIsS}
If a $\psi$-irreducible, aperiodic Markov chain 
admits a spectral gap in $\LV$ or $L_2$, then the only 
pole on the unit circle in $\Co$ is $\lambda = 1$, 
and this pole has multiplicity one.
\end{lemma}

\begin{proof}
We present the proof for $\LV$; the proof in $L_2$ is identical.

We first note that the existence of a spectral gap implies   
ergodicity:  There is a left eigenmeasure $\mu$ corresponding 
to the eigenvalue $1$, satisfying $\mu P = \mu$ 
and $|\mu|(V) = \|\mu\|_V<\infty$.   
On letting $\pi(\varble) = |\mu(\varble)|/|\mu(\state)|$ 
we conclude that $\pi$ is super-invariant:  $\pi P\ge \pi$.   
Since $\pi(\state)=1$ we must have invariance.    
The ergodic theorem for positive recurrent Markov chains 
implies that $E[G(X(n))\,|\,X(0)=x]  \to \int G\,d\pi$, as $n\to\infty$, 
whenever $G\in L_1(\pi)$.

Ergodicity rules out the existence of multiple eigenfunctions 
corresponding to $\lambda=1$.  Hence, if this pole has 
multiplicity greater than one, then there is a generalized 
eigenfunction $h\in\LV$ satisfying,
\[
Ph = h + 1.
\]
Iterating gives $P^n h\, (x) = E[ h(X(n))\,|\,X(0)=x]  
= h(x) +n$ for $n\ge 1$.   This rules out ergodicity,  
and proves that $\lambda=1$ has multiplicity one.

We now show that if $\lambda \in \clS_V$ with 
$|\lambda|=1$, then $\lambda=1$.   To see this, 
let $h\in \LV$ denote an eigenfunction,  $Ph=\lambda h$.  
Iterating, we obtain,
\[
E[ h(X(n))\,|\,X(0)=x] = h(x) \lambda^{n}.
\]
Then, letting $n\to \infty$, the right-hand-side 
converges to $\int h\,d\pi$ for a.e.\ $x$.,
so that $\lambda=1$ and $h(x)=\int h\,d\pi$,
$\pi$-a.e.
\end{proof}

Therefore, for a $\psi$-irreducible, aperiodic chain,
the existence of a spectral gap in $L_2$ is equivalent to the
existence of a single eigenvalue $\lambda=1$ on the
unit circle, which has multiplicity one. The spectral
gap $\delta_2$ is then defined as,
$$\delta_2=1-\sup\{|\la|\;:\;\lambda\in\clS_2,\,\lambda\neq 1\},$$
and similarly for $\delta_V$.

Next we state two well-known, alternative expressions
for the $L_2$-spectral gap $\delta_2$ of a reversible chain.
See, e.g., \cite[Theorem~2.1]{roberts-rosenthal:97}
and \cite[Proposition~VIII.1.11]{conway:book}.

\begin{lemma}
\label{t:L2gap}
Suppose $\Xp$ is a $\psi$-irreducible,
aperiodic, reversible Markov chain. 
Then, its $L_2$-spectral
gap $\delta_2$ admits the alternative
characterizations,
\ben
\delta_2
&=&
	1- \sup\Big\{\frac{\|\nu P\|_2}{\|\nu\|_2}\;:\;
	\mbox{signed measures $\nu$ with}\;\nu(\state)=0,\;
	\|\nu\|_2\neq 0\Big\}\\
&=&
	1-\lim_{n\to\infty}\Big(\lll P^n-\one\otimes\pi\lll_2\Big)^{1/n},
\een
where the limit is the usual spectral radius
of the semigroup $\{\haP^n\}$ 
generated by the kernel $\haP=P-\one\otimes\pi$,
acting on 
functions in $L_2(\pi)$.
\end{lemma}

A similar result holds for $\delta_V$, even in the
absence of reversibility; see, e.g., \cite{kontoyiannis-meyn:II}.

\begin{lemma}
\label{t:LVgap}
Suppose $\Xp$ is a $\psi$-irreducible,
aperiodic Markov chain. 
Then, its $L_\infty^V$-spectral
gap $\delta_V$ admits the following
alternative characterization in terms of
the spectral radius,
\ben
\delta_V
&=&
	1-\lim_{n\to\infty}\Big(\lll P^n-\one\otimes\pi\lll_V\Big)^{1/n}.
\een
\end{lemma}

\newpage

\section{Proofs}
\label{s:proofs}

First we prove Theorem~\ref{t:L2givesV}. 
The following notation will be useful
throughout this section.

For a Markov chain $\Xp=\{X(n)\}$,
the first hitting time and first return 
time to a set $C\in\clB$ 
are defined, respectively, by,
\begin{equation}\begin{aligned}
\sigma_C &\eqdef \min\{ n\ge 0 :  X(n)\in C\};
\\
\tau_C &\eqdef \min\{ n\ge 1 :  X(n)\in C\}.
\end{aligned} 
\label{e:tausigma}
\end{equation}  

Conditional on $X(0)=x$, the expectation 
operator corresponding to the measure
defining the distribution of the process
$\Xp=\{X(n)\}$ is denoted $\Expect_x(\cdot)$,
so that, for example, 
$P^nF(x)=E[F(X(n))\,|\,X(0)=x]=\Expect_x[F(X(n))]$.
For an arbitrary signed
measure $\mu$ on $(\state,\clB)$, 
we write $\mu(F)$ for 
$\int F\,d\mu$, for any function
$F:\state\to\Co$ for which
the integral exists.

\medskip


\noindent
{\em Proof of Theorem~\ref{t:L2givesV}. }
Since $\pi(h^2)<\infty$, and the chain is $\psi$-irreducible, 
it follows that there exists an increasing sequence 
of \textit{$h^2$-regular sets} providing a $\pi$-a.e.\ covering 
of $\state$ \cite[Theorem 14.2.5]{meyn-tweedie:book2}.   
That is,   there is a sequence of  sets $\{S_r : r\in\nat\}$ 
such that $\pi(S_r)\to 1$ as $r\to\infty$, 
$S_r\subset S_{r+1}$ for each $r$, and the 
following bounds hold,
\[
\begin{aligned}
V_r(x)   \eqdef\Expect_x\Bigl[\sum_{n=0}^{\tau_{S_r}} 
h^2(X(n)) \Bigr] &<\infty,\qquad\mbox{for}\;\pi\mbox{-a.e.}\; x
\\
\sup_{x\in S_r } V_r(x) &<\infty.
\end{aligned}
\]

Since the chain admits a spectral gap in $L_2$, 
combining Theorem~2.1 of \cite{roberts-rosenthal:97} 
with Lemma~\ref{t:L2gap} and the results of
\cite{roberts-tweedie:01},
we have that it is geometrically ergodic.
Hence, from \cite[Theorem~15.4.2]{meyn-tweedie:book2} it follows that 
there exists a sequence of \textit{Kendall sets} providing 
a $\pi$-a.e.\ covering of $\state$. That is, there is a sequence of 
sets $\{K_r : r\in\nat\}$ and 
positive constants 
$\{\theta_r: r\in\nat\}$
satisfying $\pi(K_r)\to 1$ as $r\to\infty$,
$K_r\subset K_{r+1}$ for each $r$, and the following
bounds hold,
\[
\begin{aligned}
U_r(x)   \eqdef\Expect_x\bigl[ \exp(\theta_r \tau_{K_r})   
	\bigr] &<\infty,\qquad \mbox{for $\pi$-a.e.}\; x\ 
\\[.2cm]
\sup_{x\in K_r }  U_r(x)   &<\infty.
\end{aligned}
\] 

We also define another collection of sets,
\[
C_{r,m}\eqdef \{ x \in \state :  U_r(x)+V_r(x)\le m\} .
\]
For each $r\ge 1$,  these sets are non-decreasing in $m$, 
and  $\pi(C_{r,m})\to 1$ as $m\to\infty$.
Moreoever,  
whenever $C_{r,m}\in \clB^+$, this set is both an $h^2$-regular set 
and a Kendall set.  This follows by combining Theorems~14.2.1 
and~15.2.1 of \cite{meyn-tweedie:book2}.   Fix $r_0$ and $m_0$ 
so that  $\pi(C_{r_0,m_0})>0$. We henceforth denote 
$C_{r_0,m_0}$ by $C$, and let $\theta>0 $ denote a value satisfying the bound,
\[
\Expect_x\bigl[ \exp(\theta \tau_{C})   \bigr] <\infty,\qquad 
\mbox{for $\pi$-a.e.}\; x,
\]
where the expectation is uniformly bounded over the Kendall set $C$.   
  
The candidate Lyapunov function can now be defined as,
\begin{equation}
V_h(x) \eqdef\Expect_x\Bigl[\sum_{n=0}^{\sigma_{C}} \bigl(1+|h(X(n))|\bigr)  
\exp\bigl(\half\theta n\bigr)  \Bigr].
\label{e:Vh}
\end{equation}
We first obtain a bound on this function. Writing,
\[
V_h(x) =\Expect_x\Bigl[\sum_{n=0}^{\sigma_{C}}   \exp\bigl(\half\theta n\bigr)   \Bigr]
+
\sum_{n=0}^{\infty} 
\Expect_x\Bigl[
 |h(X(n))|   \exp\bigl(\half\theta n\bigr) \ind\{n\le \sigma_{C}\} \Bigr],
\]
we see that the first term is finite $\pi$-a.e.\ by construction.   
The square of the
second term is bounded above, 
using the Cauchy-Shwartz inequality,
by,
\ben
\Expect_x\Bigl[ \sum_{n=0}^{\infty} 
 |h(X(n))|^2   \ind\{n\le \sigma_{C}\} \Bigr] 
  \Expect_x\Bigl[ \sum_{n=0}^{\infty}  
  \exp\bigl( \theta n\bigr) \ind\{n\le \sigma_{C}\} \Bigr] 
=U_{r_0}(x) \Expect_x\Bigl[\sum_{t=0}^{\sigma_{C}}
\exp\bigl( \theta t\bigr)  \Bigr] ,
\een
so that $V_h$ is finite $\pi$-a.e.,
and we also easily see that $|h|\le V_h$
so that $h\in L_\infty^{V_h}$.

Next we show that $V_h$ satisfies (V4):  
First apply $e^{\half\theta } P$ 
to the function $V_h$ to obtain,
\begin{equation}
e^{\half\theta } PV_h\, (x) = \Expect_x\Bigl[\sum_{n=1}^{\tau_{C}} 
\bigl(1+|h(X(n+1))|\bigr)  \exp\bigl(\half\theta (t+1)\bigr)  \Bigr] 
\label{e:PVh}
\end{equation}
We have $\tau_C=\sigma_C$ when $X(0)\in C^c$. This gives,
\[
e^{\half\theta } PV_h\, (x) =   V_h(x) -  \bigl(1+|h(x)|\bigr) ,
\qquad x\in C^c.    
\]

If $X(0)=x\in C$, then the previous arguments imply that the 
right-hand-side of \eqref{e:PVh} is finite, and in fact uniformly 
bounded over  $x\in C$.   Combining these results, we conclude that 
there exists a constant $b_0$ such that,
\[
 PV_h \le e^{-\half\theta }  V_h   + b_0\ind_C
\]
Regular sets are necessarily small
\cite[Theorem~11.3.11]{meyn-tweedie:book2}
so that this 
is a version of the drift inequality (V4).

Finally note that, by the fact that (V4) implies
the weaker drift condition (V3) of \cite{meyn-tweedie:book2},
the function $V_h$ is $\pi$-integrable by
\cite[Theorem~14.0.1]{meyn-tweedie:book2}.
\qed

\medskip

Theorem~\ref{t:L2givesV} states that (V4) holds
for a Lyapounov function $V_h$ with
$h\in L_\infty^{V_h}$.
If this could be strengthened to show that
for every geometrically ergodic chain and
any $h\in L_2$, the chain was geometrically
ergodic with a Lyapunov function $V_h$ that
had $h^2\in L_\infty^{V_h}$,
then the central limit theorem would hold for the partial 
sums of $h(X(n))$ \cite[Theorem~17.0.1]{meyn-tweedie:book2}.
But this is not generally possible:

\begin{proposition}
\label{t:Olle}
There exists a geometrically ergodic Markov chain 
on a countable state space $\state$ and a function 
$G\in L_2$ with mean
$\pi(G)=0$, for which the central limit theorem
fails in that the normalized partial sums,
\begin{equation}
\frac{1}{\sqrt n} \sum_{i=0}^{n-1} G(X(i)),
\;\;\;\;n\geq 1,
\label{e:sqrtnSn}
\end{equation} 
converge neither to a normal 
distribution nor to a point mass.
\end{proposition}

The result of the proposition
appears in \cite[Theorem~1.3]{hag05a},
and an earlier counterexample in \cite{bra83} 
yields the same conclusion.  Based on these 
counterexamples we now show that geometric 
ergodicity does not imply a spectral gap 
in the Hilbert space setting.

\medskip

\noindent
{\em Proof of Theorem~\ref{t:geoNotL2spectrum}. }
Suppose that the Markov chain $\Xp=\{X(n)\}$
constructed in Proposition~\ref{t:Olle}
does admit a spectral gap in $L_2$.   
Then its autocorrelation fuction
decays geometrically fast, 
for any $h\in L_2$: 
Assuming without loss of generality that $\pi(h)=0$, 
and letting $R_h(n) =\pi(h P^n h)$, for all $n$,
we have the bound,
\[
|R(n) |  \le \sqrt{\pi(h^2) \pi((P^n h)^2)} ,\qquad n\ge 1.
\]
Applying \Theorem{L2givesV}, we conclude that the 
right-hand-side decays geometrically fast as $n\to\infty$.  
Consequently, the sequence of normalized sums,
\[
S_n\eqdef
\frac{1}{\sqrt{n}} \sum_{i=0}^{n-1} h(X(i)),\qquad n\ge 1,
\]
is uniformly bounded in $L_2$, i.e.,
\[
\limsup_{n\to\infty} \Expect_\pi[S_n^2]  \le \sum_{n=-\infty}^\infty |R(n)|,
\]
where $\Expect_\pi[\cdot]$ denotes the expectation operator 
corresponing to the stationary version of the chain.
However, this is impossible for the choice of the function
$h=G$ as in Proposition~\ref{t:Olle}:
In \cite[p.~81]{hag05a} it is shown that the
corresponding normalized sums in \eqref{e:sqrtnSn} 
fail to define a tight sequence of probability 
distributions.  This is a consequence 
of \cite[Lemma~3.2]{hag05a}.

This contradiction establishes the claim that the Markov 
chain of \Proposition{Olle} cannot admit a spectral gap 
in $L_2$.   
\qed

\medskip

Finally we prove Propositions~\ref{t:equiv:V}
and~\ref{t:equiv:2}.

\medskip

\noindent
{\em Proof of Proposition~\ref{t:equiv:V}. }
The equivalence stated in the proposition
is obtained on combining  \Lemma{WhereIsS} with  
\cite[Proposition~4.6]{kontoyiannis-meyn:I}. 
To explain this,
we introduce new terminology:  
The transition kernel is called \textit{$V$-uniform} 
if $\lambda = 1$ is  the only pole on the unit circle in $\Co$,  
and this pole  has multiplicity one.   
Proposition~4.6 of \cite{kontoyiannis-meyn:I} states 
that geometric ergodicity with resepct to a Lyapunov
function $V$ is equivalent to $V$-uniformity  
of the kernel $P$. Consequently, the direct
part of the proposition holds,
since $V$-uniformity of $P$ implies that
it admits a spectral gap in $\LV$.
 
Conversely, if the chain admits a spectral gap in $\LV$,  
then \Lemma{WhereIsS} states that $P$ is $V$-uniform.   
Applying Proposition~4.6 of \cite{kontoyiannis-meyn:I} 
once more, we conclude that the chain is geometrically
ergodic with the same Lyapunov function $V$.
\qed

\medskip

\noindent
{\em Proof of Proposition~\ref{t:equiv:2}. }
The forward direction of the statement of the
proposition is contained in 
\cite{roberts-rosenthal:97} and
\cite{roberts-tweedie:01}.

The converse again follows from   
\Lemma{WhereIsS} and a minor modification of the arguments used in  
\cite[Proposition~4.6]{kontoyiannis-meyn:I}.   
If the chain admits a spectral gap in $L_2$, then the 
lemma states that $\lambda = 1$ has multiplicity one, 
and that this is  the only pole on the unit circle in $\Co$.  
It follows that for some $\rho<1$, the inverse 
$[zI-(P-\one\otimes\pi)]^{-1}$ exists as a bounded linear 
operator on $L_2$, whenever $|z|\ge \rho$.  Denote  
$b_\rho=\sup \lll [zI-(P-\one\otimes\pi)]^{-1}\lll_2 :  
|z|=\rho\}$, where $\lll\varble \lll_2$ is the induced 
operator norm on $L_2$. 

Following the proof of \cite[Theorem~4.1]{kontoyiannis-meyn:I}, 
we conclude that finiteness of $b_\rho$ implies a form 
of geometric ergodicity:  For any $g\in L_2$,
\[
\frac{1}{2\pi}    \int_0^{2\pi} e^{in\phi} 
[ \rho e^{in\phi}I-(P-\one\otimes\pi)]^{-1} g   
=  \rho^{-n-1}(P^n g - \pi(g)).
\]
Therefore, the $L_2$-norm of the left-hand-side 
is bounded by $ b_\rho \| g\|_2$.    
This gives,
\[
\|P^n g - \pi(g) \| _2 \le b_\rho  \| g\|_2\rho^{n+1},\qquad n\ge 1.
\]
It follows from   \cite[Theorem 15.4.3]{meyn-tweedie:book2} that the 
Markov chain is geometrically ergodic.
\qed

\bibliographystyle{plain}

\begin{thebibliography}{10}

\bibitem{bra83}
R.C. Bradley, Jr.
\newblock Information regularity and the central limit question.
\newblock {\em Rocky Mountain J. Math.}, 13(1):77--97, 1983.

\bibitem{bremaud:book}
P.~Br{\'e}maud.
\newblock {\em Markov chains: Gibbs fields, Monte Carlo simulation, and
  queues}, volume~31 of {\em Texts in Applied Mathematics}.
\newblock Springer-Verlag, New York, 1999.

\bibitem{brooks-roberts:98}
S.P. Brooks and G.O. Roberts.
\newblock Convergence assessment techniques for {M}arkov chain {M}onte {C}arlo.
\newblock {\em Statistics and Computing}, 8:319--335, 1998.

\bibitem{conway:book}
J.B. Conway.
\newblock {\em A course in functional analysis}.
\newblock Springer-Verlag, New York, second edition, 1990.

\bibitem{deljun99}
M.~Dellnitz and O.~Junge.
\newblock On the approximation of complicated dynamical behavior.
\newblock {\em SIAM J. on Numerical Analysis}, 36(2):491--515, 1999.

\bibitem{diaconis:09}
P.~Diaconis.
\newblock The {M}arkov chain {M}onte {C}arlo revolution.
\newblock {\em Bull. Amer. Math. Soc. (N.S.)}, 46(2):179--205, 2009.

\bibitem{diaconis:08}
P.~Diaconis, K.~Khare, and L.~Saloff-Coste.
\newblock Gibbs sampling, exponential families and orthogonal polynomials.
\newblock {\em Statist. Sci.}, 23(2):151--200, 2008.
\newblock With comments and a rejoinder by the authors.

\bibitem{diaconis-SC:98}
P.~Diaconis and L.~Saloff-Coste.
\newblock What do we know about the {M}etropolis algorithm?
\newblock {\em J. Comput. System Sci.}, 57(1):20--36, 1998.
\newblock 27th Annual ACM Symposium on the Theory of Computing (STOC'95) (Las
  Vegas, NV).

\bibitem{hag05a}
O.~H{\"a}ggstr{\"o}m.
\newblock On the central limit theorem for geometrically ergodic {M}arkov
  chains.
\newblock {\em Probab. Theory Related Fields}, 132(1):74--82, 2005.

\bibitem{hag05b}
O.~H{\"a}ggstr{\"o}m.
\newblock Acknowledgement of priority concerning ``{O}n the central limit
  theorem for geometrically ergodic {M}arkov chains'' [{P}robab. {T}heory
  {R}elated {F}ields {\bf 132} (2005), no. 1, 74--82.].
\newblock {\em Probab. Theory Related Fields}, 135(3):470, 2006.

\bibitem{huimeysch04a}
W.~Huisinga, S.P. Meyn, and C.~Schuette.
\newblock Phase transitions and metastability in {Markovian} and molecular
  systems.
\newblock {\em Ann. Appl. Probab.}, 2001.
\newblock {to appear}.

\bibitem{kontoyiannis-meyn:I}
I.~Kontoyiannis and S.P. Meyn.
\newblock Spectral theory and limit theorems for geometrically ergodic {M}arkov
  processes.
\newblock {\em Ann. Appl. Probab.}, 13:304--362, February 2003.

\bibitem{kontoyiannis-meyn:II}
I.~Kontoyiannis and S.P. Meyn.
\newblock Large deviation asymptotics and the spectral theory of
  multiplicatively regular {M}arkov processes.
\newblock {\em Electron. J. Probab.}, 10(3):61--123, 2005.

\bibitem{diaconis:07}
G.~Lebeau and P.~Diaconis.
\newblock M\'etropolis: le jour o\`u l'\'etoile probabilit\'e entra dans le
  champ gravitationnel de la galaxie microlocale.
\newblock In {\em S\'eminaire: \'{E}quations aux {D}\'eriv\'ees {P}artielles.
  2006--2007}, S\'emin. \'Equ. D\'eriv. Partielles, pages Exp. No. XIV, 13.
  \'Ecole Polytech., Palaiseau, 2007.

\bibitem{meyn-tweedie:book2}
S.~P. Meyn and R.~L. Tweedie.
\newblock {\em {Markov Chains and Stochastic Stability}}.
\newblock Cambridge University Press, London, 2nd edition, 2009.
\newblock Published in the Cambridge Mathematical Library. 1993 edition online:
  {\tt http://black.csl.uiuc.edu/{\~{ }}meyn/pages/book.html}.

\bibitem{meyn-tweedie:94b}
S.P. Meyn and R.L. Tweedie.
\newblock Computable bounds for geometric convergence rates of {M}arkov chains.
\newblock {\em Ann. Appl. Probab.}, 4(4):981--1011, 1994.

\bibitem{tetali:06}
R.~Montenegro and P.~Tetali.
\newblock Mathematical aspects of mixing times in {M}arkov chains.
\newblock {\em Found. Trends Theor. Comput. Sci.}, 1(3), 2006.

\bibitem{nummelin:book}
E.~Nummelin.
\newblock {\em General Irreducible {M}arkov Chains and Nonnegative Operators}.
\newblock Cambridge University Press, Cambridge, 1984.

\bibitem{roberts-rosenthal:97}
G.O. Roberts and J.S. Rosenthal.
\newblock Geometric ergodicity and hybrid {M}arkov chains.
\newblock {\em Electron. Comm. Probab.}, 2:no.\ 2, 13--25 (electronic), 1997.

\bibitem{roberts-tweedie:01}
G.O. Roberts and R.L. Tweedie.
\newblock Geometric {$L\sp 2$} and {$L\sp 1$} convergence are equivalent for
  reversible {M}arkov chains.
\newblock {\em J. Appl. Probab.}, 38A:37--41, 2001.
\newblock Probability, statistics and seismology.

\bibitem{ros95b}
J.S. Rosenthal.
\newblock Correction: ``{M}inorization conditions and convergence rates for
  {M}arkov chain {M}onte {C}arlo''.
\newblock {\em J. Amer. Statist. Assoc.}, 90(431):1136, 1995.

\bibitem{ros95a}
J.S. Rosenthal.
\newblock Minorization conditions and convergence rates for {M}arkov chain
  {M}onte {C}arlo.
\newblock {\em J. Amer. Statist. Assoc.}, 90(430):558--566, 1995.

\bibitem{winkler:book}
G.~Winkler.
\newblock {\em Image Analysis, Random Fields and Dynamic Monte Carlo Methods: A
  Mathematical Introduction}.
\newblock Springer-Verlag, Berlin, 1995.

\end{thebibliography}

\end{document}